\pdfoutput=1
\documentclass[a4paper,10pt]{amsart}
\usepackage{
 amsmath,
 amssymb,
 fullpage,
 mathtools,
 thmtools,
 enumerate
}
\usepackage[
 pagebackref,
 colorlinks,
 allcolors=blue
]{hyperref}

\usepackage[capitalize, noabbrev]{cleveref}

\renewcommand{\sp}{\operatorname{sp}}
\newcommand{\As}{\operatorname{As}}
\newcommand{\St}{\mathrm{St}}

\newlength{\bibitemsep}
\newlength{\bibparskip}
\let\oldthebibliography\thebibliography
\renewcommand\thebibliography[1]{%
 \oldthebibliography{#1}%
 \setlength{\parskip}{\bibparskip}%
 \setlength{\itemsep}{\bibitemsep}%
}

\renewcommand{\AA}{\mathbf{A}}
\newcommand{\QQ}{\mathbf{Q}}
\newcommand{\CC}{\mathbf{C}}

\newcommand{\cI}{\mathcal{I}}
\newcommand{\cO}{\mathcal{O}}
\newcommand{\cS}{\mathcal{S}}

\newcommand{\vep}{\varepsilon}

\DeclareMathOperator{\Hom}{Hom}
\DeclareMathOperator{\Frob}{Frob}
\DeclareMathOperator{\GL}{GL}
\DeclareMathOperator{\SO}{SO}
\DeclareMathOperator{\GSp}{GSp}
\DeclareMathOperator{\PGL}{PGL}

\usepackage{bbold}
\newcommand{\triv}{\mathbb{1}}

\newtheorem{definition}{Definition}[section]
\newtheorem{theorem}[definition]{Theorem}
\newtheorem{proposition}[definition]{Proposition}
\newtheorem{lemma}[definition]{Lemma}
\newtheorem*{mainthm}{Main Theorem}

\newtheorem{subtheorem}{Theorem}[theorem]

\theoremstyle{remark}
\declaretheorem[name=Remark,sibling=theorem,qed={\lower-0.3ex\hbox{$\diamond$}}]{remark}

\begin{document}

\title{Gross--Prasad periods for reducible representations}
\author{David Loeffler}
\thanks{Supported by Royal Society University Research Fellowship UF160511.}
\address{Warwick Mathematics Institute, University of Warwick, Coventry CV4 7AL, UK}
\email{d.a.loeffler@warwick.ac.uk}

\begin{abstract}
 We study $\GL_2(F)$-invariant periods on representations of $\GL_2(A)$, where $F$ is a nonarchimedean local field and $A/F$ a product of field extensions of total degree 3. For irreducible representations, a theorem of Prasad shows that the space of such periods has dimension $\le 1$, and is non-zero when a certain $\vep$-factor condition holds. We give an extension of this result to a certain class of reducible representations (of Whittaker type), extending results of Harris--Scholl when $A$ is the split algebra $F \times F \times F$.
\end{abstract}

\maketitle

\section{Introduction}

 One of the central problems in the theory of smooth representations of reductive groups over nonarchimdean local fields is to determine when a representation of a group $G$ admits a linear functional invariant under a closed subgroup $H$ (an $H$-invariant period).

 The Gross--Prasad conjectures \cite{grossprasad92} give a very precise and elegant description of when such periods exist, for many natural pairs $(G, H)$, in terms of $\vep$-factors. However, the original formulation of these conjectures applies to members of \emph{generic} $L$-packets for $G$; and the analogous picture for representations in non-generic $L$-packets is rather more complex. Although the $\vep$-factor is still well-defined for all such $L$-packets, the conjecture formulated in \cite{gangrossprasad20} only applies when the $L$-parameters satisfy an additional ``relevance'' condition, raising the natural question of whether the $\vep$-factors for non-relevant $L$-packets have any signficance in terms of invariant periods.

 In this short note, we describe some computations of branching laws in the following simple case: $G$ is $\GL_2(A)$, where $A/F$ is a cubic \'etale algebra, and $H$ is the subgroup $\GL_2(F)$. Our computations suggest an alternative approach to the theory: rather than studying branching laws for non-generic irreducible representations, we focus on representations which are possibly reducible, but satisfy a certain ``Whittaker type'' condition. We show that $H$-invariant periods on these representations are unique if they exist, and that their existence is governed by $\vep$-factors, extending the results of Prasad \cite{prasad90,prasad92} for irreducible generic representations, and Harris--Scholl \cite{harrisscholl01} for $A$ the split algebra (in which case the $\vep$-factor is always $+1$). In this optic, the ``relevance'' condition appears as a criterion for the $H$-invariant period to factor through the unique irreducible quotient.

 This result, combined with other recent works such as that of Chan \cite{chan21} in the case $(G, H) = (\GL_n(F) \times \GL_{n+1}(F), \GL_n(F))$, would seem to suggest that many ``Gross--Prasad-style'' branching results should extend to Whittaker-type representations, and we hope to explore this further in future works.

 We conclude with an application to global arithmetic. For $\pi$ a Hilbert modular form over a real quadratic field, the constructions of \cite{Kings-Higher-Regulators, leiloefflerzerbes18, grossi19} give rise to a family of cohomology classes taking values in the 4-dimensional Asai Galois representation associated to $\pi$. We show that if $\pi$ is not of CM type and not a base-change from $\QQ$, then these elements all lie in a 1-dimensional subspace. This is the analogue for quadratic Hilbert modular forms of the result proved in \cite{harrisscholl01} for Beilinson's elements attached to the Rankin convolution of two modular forms.

 \subsubsection*{Acknowledgements} I am grateful to Giada Grossi and Dipendra Prasad for interesting conversations in connection with this paper, and especially to Nadir Matringe for his answer to a question of mine on MathOverflow, which provided the key to Theorem \ref{thm:asaiA}. I would also like to thank Kei Yuen Chan, for pointing out the relevance of a result of Moeglin--Waldspurger recalled in \S \ref{sect:GGP} below; and the anonymous referee, for suggesting a much cleaner proof of Theorem \ref{thm:asaiB}.
\section{Statements}

 Throughout this paper, $F$ denotes a non-archimedean local field of characteristic 0. If $G$ is a reductive group over $F$, then a ``representation'' of $G(F)$ shall mean a smooth linear representation on a complex vector space.

 \subsection{Epsilon-factors}

  We choose a nontrivial additive character $\psi$ of $F$. For Weil--Deligne representations $\rho$ of $F$, we define epsilon-factors $\vep(\rho) = \vep(\rho, \psi)$ following Langlands (the ``$\vep_L$'' convention in \cite[\S 3.6]{tate79}), so that $\vep(\rho)$ is independent of $\psi$ if $\det(\rho) = 1$. We note that
  \begin{align*}
   \vep(\rho_1 \oplus \rho_2) &= \vep(\rho_1) \vep(\rho_2), &
   \vep(\rho) \vep(\rho^\vee) &= (\det \rho)(-1)
  \end{align*}
  where $\det(\rho)$ is identified with a character of $F^\times$ via class field theory.

  We write $\sp(n)$ for the $n$-dimensional Weil--Deligne representation given by the $(n-1)$-st symmetric power of the Langlands parameter of the Steinberg representation, so that the eigenvalues of the Frobenius element on $\sp(n)$ are $q^{\tfrac{1-n}{2}}, q^{\tfrac{3-n}{2}}, \dots, q^{\tfrac{n-1}{2}}$ where $q$ is the size of the residue field.

 \subsection{The generic Langlands correspondence for $\GL_2$}

  The classical local Langlands correspondence for $\GL_2$ is a bijection between irreducible smooth representations of $\GL_2(F)$, and 2-dimensional Frobenius-semisimple representations of the Weil--Deligne group of $F$.

  In this paper, we will use the following modification of the correspondence. A representation of $\GL_2(F)$ is said to be \emph{of Whittaker type} if it is either irreducible and generic, or a reducible principal series representation with 1-dimensional quotient. (These are precisely the representations of $\GL_2(F)$ which have well-defined Whittaker models.) The \emph{generic Langlands correspondence} is a bijection between Whittaker-type representations of $\GL_2(F)$ and 2-dimensional Frobenius-semisimple Weil--Deligne representations; it agrees with the classical Langlands correspondence on irreducible generic representations, and maps a reducible Whittaker-type principal series to the classical Langlands parameter of its 1-dimensional quotient.\footnote{This correspondence was introduced in \cite{breuilschneider07}; but our conventions differ from \emph{op.cit.} by a power of the norm character, in order that our generic Langlands correspondence extend the classical one.}

  In particular, the unramified Weil--Deligne representation with Frobenius acting as $\left(\begin{smallmatrix} q^{1/2} \\ & q^{-1/2}\end{smallmatrix}\right)$ corresponds to the reducible principal series $\Sigma_F$ containing the Steinberg representation $\St_F$ as subrepresentation and trivial 1-dimensional quotient. (We omit the subscript $F$ if it is clear from context.)

 \subsection{Statement of the theorem}

  We now state our main theorem. Let $A / F$ be a separable cubic algebra, so $A$ is a product of field extensions of $F$ of total degree 3. Let $\omega_A$ be the quadratic character of $F^\times$ determined by the class of $\operatorname{disc}(A)$ in $F^\times / F^{\times 2}$. We let $G = \GL_2(A)$, and $H = \GL_2(F)$, embedded in $G$ in the obvious way.

  The Langlands dual group of $\GL_2 / A$ has a natural 8-dimensional \emph{Asai}, or \emph{multiplicative induction}, representation; in the case $A = F^3$ this is simply the tensor product of the defining representations of the factors. We use this representation, and the generic Langlands correspondence for $\GL_2$ above, to define Asai epsilon-factors $\vep(\As(\Pi))$ for Whittaker-type representations of $\GL_2(A)$.

  Finally, we consider Jacquet--Langlands transfers. Let $H' = D^\times$ where $D/F$ is the unique non-split quaternion algebra. Let $G' = (D \otimes_F A)^\times$, and let $\Pi'$ be the Jacquet--Langlands transfer of $\Pi$ to $G'$ if this exists, and 0 otherwise.

  \begin{remark}
   Note that if $A = E \times F$ for $E$ a quadratic field extension, then $D^\times$ is split over $E$, and hence $G' = \GL_2(E) \times D^\times(F)$. Thus if $\Pi = \pi \boxtimes \sigma$, for $\pi$, $\sigma$ representations of $\GL_2(E)$ and $\GL_2(F)$ respectively, we have $\Pi' = \pi \boxtimes \sigma'$. In particular, $\Pi' \ne 0$ whenever $\sigma'$ is discrete series (even if $\pi$ is principal series, possibly reducible).
  \end{remark}

  \begin{mainthm}
   Let $\Pi$ be a representation of $\GL_2(A)$ of Whittaker type, whose central character is trivial on $F^\times$ (embedded diagonally in $A^\times$). Then we have
   \[ \dim \Hom_{H}(\Pi, \triv) = \begin{cases}
    1 & \text{if $\vep(\As(\Pi)) \omega_A(-1) = 1$}, \\
    0 & \text{if $\vep(\As(\Pi)) \omega_A(-1) = -1$};
    \end{cases}
   \]
   and
   \[
    \dim \Hom_{H}(\Pi, \triv) + \dim \Hom_{H'}(\Pi', \triv) = 1. \]
  \end{mainthm}

  If $\Pi$ is an irreducible generic representation, then this is the main result of \cite{prasad90} for $A$ the split algebra, and \cite{prasad92} for non-split $A$ (modulo the case of supercuspidal representations of cubic fields, completed in \cite{prasadSP08}). The new content of the above theorem is that this also holds for reducible Whittaker-type $\Pi$.

  \begin{remark}
   Any such $\Pi$ can be written as the specialisation at $s = 0$ of an analytic family of Whittaker-type representations $\Pi(s)$ indexed by a complex parameter $s$, which are irreducible for generic $s$ and all have central character trivial on $F^\times$. For such families, the $\vep$-factors $\vep(\As \Pi(s))$ are locally constant as a function of $s$; hence, given the results of \cite{prasad92, prasadSP08} in the irreducible case, our theorem is equivalent to the assertion that $\dim \Hom_H(\Pi(s), \triv)$ and $\dim \Hom_{H'}(\Pi(s)', \triv)$ are locally constant in $s$.
  \end{remark}

 \subsection{Relation to results of Moeglin--Waldspurger}\label{sect:GGP}

  The Proposition in section 1.3 of \cite{moeglinwaldspurger12} gives a formula for branching multiplicities for certain parabolically-induced representations of special orthogonal groups $\SO(d) \times \SO(d')$ (with $d - d'$ odd), expressing these in terms of multiplicities for irreducible tempered representations of smaller special orthogonal groups. These results are applied in \emph{op.cit.}~to prove the Gross--Prasad conjecture for irreducible representations in non-tempered generic $L$-packets (by reduction to the tempered case); but the results are also valid for reducible representations.

  Since the split form of $\SO(3)$ is $\PGL(2)$, and $\SO(4)$ is closely related to $\PGL(2) \times \PGL(2)$, one can derive many cases of our Main Theorem from their result applied to various forms of $\SO(3) \times \SO(4)$. In fact, if $A = F^3$ or $A = E \times F$ for $E$ quadratic, we can obtain in this way all cases of the Main Theorem not already covered by Prasad's results.

  However, the case when $A$ is a cubic field extension does not appear to fit into the framework of \emph{op.cit.}; and the proof given in \emph{op.cit.} is rather indirect, particularly in the case when the $\SO(3)$ representation is reducible, in which case their argument requires a delicate switch back and forth between representations of $\SO(3)\times \SO(4)$ and $\SO(4)\times \SO(5)$. So we hope that the alternative, more direct approach given here will be of interest.

\section{Split triple products}

 We first put $A = F \times F \times F$.

 \begin{theorem}[Prasad, Harris--Scholl] Let $\pi_1$, $\pi_2$, $\pi_3$ be representations of $\GL_2(F)$ of Whittaker type, with central characters $\omega_i$ such that $\omega_1 \omega_2 \omega_3 = 1$. Then we have
 \[ \dim \Hom_{\GL_2(F)}(\pi_1 \otimes \pi_2 \otimes \pi_3, \triv) =
   \begin{cases}
    1 & \text{if $\vep(\pi_1 \times \pi_2 \times \pi_3) = +1$},\\
    0 & \text{if $\vep(\pi_1 \times \pi_2 \times \pi_3) = -1$},\\
   \end{cases} \]
  and
  \[ \dim \Hom_{\GL_2(F)}(\pi_1 \otimes \pi_2 \otimes \pi_3, \triv) + \dim \Hom_{D^\times(F)}(\pi'_1 \otimes \pi'_2 \otimes \pi'_3, \triv) = 1.\]
 \end{theorem}

 If the $\pi_i$ are all irreducible, then the above is the main result of \cite{prasad90}. If one or more of the $\pi_i$ is isomorphic to a twist of $\Sigma_F$, then the $\vep$-factor is automatically $+1$, and $\pi'_1 \otimes \pi'_2 \otimes \pi'_3$ is the zero representation. So all that remains to be shown is that in this case we have $\dim \Hom_{\GL_2(F)}(\pi_1 \otimes \pi_2 \otimes \pi_3, \triv)= 1$.  This is estabished in Propositions 1.5, 1.6 and 1.7 of \cite{harrisscholl01}, except for one specific case, which is when all three of the $\pi_i$ are twists of $\Sigma$ by characters.

 In this case, by twisting we may assume $\pi_2 = \pi_3 = \Sigma$ and $\pi_1 = \Sigma \otimes \eta$ where $\eta$ is a character of $F^\times$ with $\eta^2 = 1$. The case $\eta = 1$ is covered by Proposition 1.7 of \emph{op.cit.}, so we assume $\eta$ is a nontrivial quadratic character. In this case $\Hom_H( \eta \otimes \Sigma_F \otimes \Sigma_F, \triv) = \Hom_H(\Sigma_F, \Sigma_F^\vee \otimes \eta) = 0$, so $\Hom_H(\pi_1 \otimes \pi_2 \otimes \pi_3, \triv)$ injects into $\Hom_H( \eta \St_F \otimes \Sigma_F \otimes \Sigma_F, \triv)$, which has dimension 1 by \cite[Proposition 1.6]{harrisscholl01}. Thus $\Hom_H(\pi_1 \otimes \pi_2 \otimes \pi_3, \triv)$ has dimension $\le 1$. Since one can easily write down a nonzero element of this space using the Rankin--Selberg zeta integral, we conclude that its dimension is 1 as required.

\section{Quadratic fields}

 We now suppose $A = E \times F$ with $E/F$ quadratic, so $\Pi = \pi \boxtimes \sigma$ for Whittaker-type representations $\pi$ of $\GL_2(E)$ and $\sigma$ of $\GL_2(F)$ such that $\omega_{\pi}|_{F^\times} \cdot \omega_{\sigma} = 1$. Since the case of $\pi$, $\sigma$ irreducible is proved in \cite{prasad92}, it suffices to consider the following cases:
 \begin{enumerate}[(a)]
 \item $\pi$ is irreducible and $\sigma = \Sigma_F$.
 \item $\sigma$ is irreducible and $\pi = \Sigma_E$;
 \item $\pi = \Sigma_E$ and $\sigma = \Sigma_F \otimes \eta$, where $\eta$ is a quadratic character.
 \end{enumerate}

 In cases (a) and (c), we always have $\vep(\As(\pi) \times \sigma)\vep_{E/F}(-1) = 1$, and $\sigma' = \{0\}$, so the Main Theorem amounts to the assertion that $\dim \Hom_H(\pi \boxtimes \sigma, \triv) = 1$. In case (b), both signs can occur.

 \addtocounter{theorem}{1}
 \begin{subtheorem}
  \label{thm:asaiA}
  Let $\pi$ be an irreducible generic representation of $\GL_2(E)$ such that $\omega_{\pi} |_{F^\times} = 1$. Then we have $\dim \Hom_H(\pi \boxtimes \Sigma_F, \triv) = 1$.
 \end{subtheorem}

 \begin{remark}
  Note that the case when $E / F$ is unramified, and $\pi$ is unramified and tempered, is part of \cite[Theorem 4.1.1]{grossi19}. However, the proof of this statement \emph{loc.cit.}~has a minor error which means the argument does not work when $\pi$ is the normalised induction of the trivial character of $B_E$. So the argument below fixes this small gap.
 \end{remark}

 \begin{proof}

  We first observe that $\Hom_H(\pi \boxtimes \Sigma_F, \triv)$ is non-zero. Since $\pi$ is generic, it has a Whittaker model $\mathcal{W}(\pi)$ with respect to any non-trivial additive character of $E$. We may suppose that this additive character is trivial on $F$, so that we may define the Asai zeta-integral
  \[ Z(W, \Phi, s) = \int_{N_H \backslash H} W(h)\Phi( (0, 1) h) |\det h|^{s}\ \mathrm{d}h,\]
  for $W \in \mathcal{W}(\pi)$ and $\Phi \in \cS(F^2)$ (the space of Schwartz functions on $F$). Here $N_H$ is the upper-triangular unipotent subgroup of $H$.

  It is well known that this integral converges for $\Re(s) \gg 0$ and has meromorphic continuation to the whole complex plane; and the values of $Z(-, -, s)$ span a nonzero fractional ideal of $\CC[q^s, q^{-s}]$, generated by an $L$-factor independent of $\Phi$ and $W$, which is the Asai $L$-factor $L(\As(\pi), s)$. Thus the map
  \[ \tag{\dag} (W, \Phi) \mapsto \lim_{s \to 0} \frac{Z(W, \Phi, s)}{L(\As(\pi), s)} \]
  defines a non-zero, $H$-invariant bilinear form $W(\pi) \otimes \cS(F^2) \to \CC$. Since the maximal quotient of $\cS(F^2)$ on which $F^\times$ acts trivially is isomorphic to $\Sigma_F$ (see e.g.~\cite[Proposition 3.3(b)]{loeffler-zeta1}), this shows that $\Hom_H(\pi \boxtimes \Sigma_F, \triv) \ne 0$ as claimed.

  So, to prove \cref{thm:asaiA}, it suffices to show that $\dim \Hom_H(\pi \boxtimes \Sigma_F, \triv) \le 1$. As $\pi$ has unitary central character, it is either a discrete-series representation, in which case it is automatically tempered, or an irreducible principal series, which may or may not be tempered. We shall consider these cases separately.

  Theorem 1.1 of \cite{AKT04} states that if $\pi$ is an irreducible tempered representation of $\GL_2(E)$, then we have $\dim \Hom_{M(F)}(\pi, \triv) = 1$, where $M(F) = \{ \left(\begin{smallmatrix} \star&\star\\ 0 & 1\end{smallmatrix}\right)\}$ is the mirabolic subgroup of $\GL_2(F)$. If we assume $\omega_{\pi} |_{F^\times} = 1$, then since $F^\times \cdot M(F) = B(F)$ is the Borel subgroup of $\GL_2(F)$, we have
  \[ \Hom_{M(F)}(\pi, \triv) = \Hom_{B(F)}(\pi, \triv) = \Hom_{H}(\pi, \operatorname{Ind}_{B(F)}^{H}(\triv)).\]
  As $\operatorname{Ind}_{B(F)}^{H}(\triv) = \Sigma_F^\vee$, this proves \cref{thm:asaiA} for tempered $\pi$.

  We now consider the principal-series case. For $\alpha, \beta$ smooth characters of $E^\times$, we write $I_E(\alpha, \beta)$ for the normalised induction to $\GL_2(E)$ of the character $\alpha \boxtimes \beta$ of $B(E)$. Note that this representation is tempered if and only if $\alpha$ and $\beta$ are unitary. We suppose $\alpha/\beta \ne |\cdot|_E^{\pm 1}$ and $\alpha\beta|_{F^\times} = 1$. Then we have the following results:
  \begin{itemize}
  \item $\Hom_H(\pi \boxtimes \St_F, \triv)$ is zero if $\alpha \beta^c = 1$, and one-dimensional otherwise, where $\beta^c$ denotes the character $x \mapsto \beta(x^c)$. See \cite[Remark 4.1.1]{prasad92}.
  \item $\Hom_H(\pi \boxtimes \triv, \triv)$ is one-dimensional if $\alpha \beta^c = 1$, or if $\alpha|_{F^\times} = \beta|_{F^\times} = 1$; otherwise it is zero. See \cite[Theorem 5.2]{matringe11}.
  \end{itemize}

  We conclude that exactly one of $\Hom_H(\pi \boxtimes \St_F, \triv)$ and $\Hom_H(\pi \boxtimes \triv, \triv)$ is nonzero (and \cref{thm:asaiA} therefore follows), \emph{unless} $\pi$ is of the form $I_E(\alpha, \beta)$ with $\alpha|_{F^\times} = \beta|_{F^\times} = 1$ and $\alpha\beta^c \ne 1$. However, in this exceptional case $\alpha$ and $\beta$ are unitary, and thus $\pi$ is tempered, so \cref{thm:asaiA} has already been established for $\pi$ above. This completes the proof of \cref{thm:asaiA}.
  \end{proof}

  \begin{remark}
   It follows, in particular, that for a generic irreducible representation $\pi$ of $\GL_2(E)$, we have $\Hom_H(\pi, \triv) \ne 0$ (i.e.~$\pi$ is ``$F$-distinguished'') if and only if the zeta-integral $(\dag)$ factors through the 1-dimensional quotient of $\Sigma_F$, and thus vanishes on all $\Phi$ with $\Phi(0, 0) = 0$; that is, $s = 0$ is an \emph{exceptional pole} of the Asai $L$-factor. This is the $n = 2$ case of a theorem due to Matringe \cite[Theorem 3.1]{matringe10} applying to $\GL_n(E)$-representations. See \cite{loeffler-zeta1} for analogous results and conjectures regarding poles of zeta-integrals for $\GSp_4$ and $\GSp_4 \times \GL_2$.
  \end{remark}

  For case (b) of the main theorem, we need the following lemma:

  \begin{lemma}
   Let $\sigma$ be an irreducible generic representation of $\GL_2(F)$ with $\omega_{\sigma} = 1$. Then $\vep(\As(\Sigma_E) \times \sigma) = \vep(\sigma) \vep(\sigma \times \omega_{E/F})$. Moreover, if $\sigma \ne \St_F$, then we have
   \[ \vep(\sigma) \vep(\sigma \times \omega_{E/F}) = \vep(\As(\St_E) \times \sigma), \]
   while for $\sigma = \St_F$ we have $\vep(\As(\St_E) \times \St_F)\omega_{E/F}(-1) = 1$ and $\vep(\As(\Sigma_E) \times \St_F)\omega_{E/F}(-1) = -1$.
  \end{lemma}

  \begin{proof}
   If $\sigma$ is not a twist of Steinberg, then its Weil--Deligne representation has trivial monodromy action, so we compute that
   \[ \vep(\As(\St_E) \times \sigma) = \vep( (\sp(3) \oplus \omega_{E/F}) \times \sigma) = \vep(\sigma \times \omega_{E/F}) \vep(\sigma)^3 \det(-\Frob: \rho_{\sigma}^{I_F})^2.\]
   Since $\sigma$ has trivial central character, $\vep(\sigma) = \pm 1$. If $\sigma$ is supercuspidal we are done, since in this case $\rho_{\sigma}^{I_F} = 0$. If $\sigma$ is principal series, then $\rho_{\sigma}^{I_F}$ must be either 0, or all of $\rho_\sigma$, since $\rho_{\sigma}$ has determinant 1. Thus $\det(-\Frob: \rho_{\sigma}^{I_F}) = 1$, so $\vep(\As(\St_E) \times \sigma) = \vep(\sigma) \vep(\sigma \times \omega_{E/F})$, proving the claim in this case. The case when $\sigma$ is a twist of the Steinberg by a nontrivial (necessarily quadratic) character can be computed similarly.
  \end{proof}

  \setcounter{theorem}{1}
  \setcounter{subtheorem}{1}
  \begin{subtheorem}
   \label{thm:asaiB}
   Let $\sigma$ be an irreducible generic representation of $\GL_2(F)$ with $\omega_{\sigma} = 1$. Then
   \begin{enumerate}[(i)]
    \item If $\vep(\sigma) \vep(\sigma \times \omega_{E/F}) = \omega_{E/F}(-1)$, then $\dim \Hom_H(\Sigma_E \boxtimes \sigma, \triv) = 1$ and $\Hom_{H'}(\Sigma_E \boxtimes \sigma', \triv) = 0$.
    \item If $\vep(\sigma) \vep(\sigma \times \omega_{E/F}) = -\omega_{E/F}(-1)$, then $\Hom_H(\Sigma_E \boxtimes \sigma, \triv) = 0$ and $\dim \Hom_{H'}(\Sigma_E \boxtimes \sigma', \triv) = 1$.
   \end{enumerate}
  \end{subtheorem}

  \begin{proof}
   We first consider the situation for $H'$. This case is relatively simple, since $H'$ is compact modulo centre, and hence the functor of $H'$-invariants is exact on the category of $H'$-representations trivial on $F^\times$. So we have
   \[ \dim \Hom_{H'}(\Sigma_E \otimes \sigma', \triv) = \dim \Hom_{H'}(\sigma', \triv) + \dim \Hom_{H'}(\St_E \otimes \sigma', \triv).\]
   Using Prasad's results for $\Hom_{H'}(\St_E \otimes \sigma', \triv)$ and the preceding lemma, we see that $\dim \Hom_{H'}(\Sigma_E \otimes \sigma', \triv)$ has dimension 1 if $\vep(\sigma) \vep(\sigma \times \omega_{E/F}) = -\omega_{E/F}(-1)$ and is zero otherwise, as required.
   \medskip

   For the group $H$, the situation is a little more complicated: since $\sigma$ is generic, we have $\Hom_H(\sigma, \triv)$ is zero, and hence there is an exact sequence
   \[ 0 \to \Hom_H(\Sigma_E \otimes \sigma, \triv) \to \Hom_H(\St_E \otimes \sigma, \triv) \to \operatorname{Ext}^1_{\PGL_2(F)}(\sigma, \triv).\]

   \emph{Claim}: The group $\operatorname{Ext}^1_{\PGL_2(F)}(\sigma, \triv)$ is 1-dimensional if $\sigma = \St_F$, and zero otherwise.\medskip

   \emph{Proof of Claim}: If $\sigma$ is supercuspidal the result is immediate, since $\sigma$ is projective in the category of $\PGL_2(F)$-representations. The remaining cases can be handled directly using Frobenius reciprocity, or alternatively, one can appeal to Schneider--Stuhler duality (as reformulated in \cite[Theorem 2]{noriprasad20}) to show that the $\operatorname{Ext}$ group is dual to $\Hom_H(D(\triv), \sigma)$ where $D$ is the Aubert--Zelevinsky involution, which sends $\triv$ to $\St_F$.\medskip

   This gives the desired formula for $\dim \Hom_H(\Sigma_E \otimes \sigma, \triv)$ in all cases except when $\sigma = \St_F$, in which case we must show that the non-trivial $H$-invariant period of $\St_E \otimes \St_F$ does not lift to $\Sigma_E \otimes \St_F$. This can be done directly: we can compute $\Sigma_E |_{\GL_2(F)}$ via Mackey theory, using the two orbits of $H$ on $\mathbf{P}^1(E)$ to obtain the exact sequence
    \[ 0 \to \operatorname{cInd}_{E^\times}^{\GL_2(F)}(\triv) \to \Sigma_E \to I_F(|\cdot|_F, |\cdot|_F^{-1}) \to 0.\]
   The latter representation is irreducible and has no homomorphisms to $\St_F$; and we saw in the proof of Theorem \ref{thm:asaiA} that
   \[ \Hom_H(\operatorname{cInd}_{E^\times}^{H}(\triv) \otimes \St_F, \triv) = \Hom_{E^\times}(\St_F, \triv) = 0.\]
   This shows that $\Hom_H(\Sigma_E \boxtimes \St_F, \triv) = 0$, completing the proof.
  \end{proof}

  \begin{remark}
   We are grateful to the anonymous referee for pointing out the significance of the vanishing of $\operatorname{Ext}^1_{\PGL_2(F)}(\sigma, \triv)$; the original version of this paper used a different and rather more complicated argument.
  \end{remark}

  \addtocounter{theorem}{-1}
  \addtocounter{subtheorem}{2}

  \begin{subtheorem}
   Let $\eta$ be a quadratic character of $F^\times$ (possibly trivial). Then we have $\dim \Hom_H(\Sigma_E \boxtimes \Sigma_F, \eta) = 1$.
  \end{subtheorem}

  \begin{proof}[Proof]
   The computation of the epsilon-factor is immediate; and by a zeta-integral argument as before, we can show that $\Hom_H(\Sigma_E \boxtimes \Sigma_F, \eta) \ne 0$ (since the representation $\Sigma_E$, despite being reducible, has a well-defined Whittaker model). So it suffices to show that the hom-space has dimension $\le 1$.

   If $\eta$ is not the trivial character, then $\Hom_H(\triv \boxtimes \Sigma_F, \eta) = 0$, so the desired Hom-space injects into $\Hom_H(\St_E \boxtimes \Sigma_F, \eta)$, which is 1-dimensional by Theorem \ref{thm:asaiA}.

   If $\eta$ is trivial, then we have seen above that $\Hom_H(\Sigma_E \boxtimes \St_F, \triv)$ is zero. So $\Hom_H(\Sigma_E \boxtimes \Sigma_F, \triv) = \Hom_H(\Sigma_E, \triv)$. From the Mackey decomposition of $\Sigma_E |_{\GL_2(F)}$ above, one sees easily that this space is 1-dimensional.
  \end{proof}

\section{Cubic fields}

 We briefly discuss the case where $A$ is a cubic extension of $F$.

 \begin{theorem}
  Let $\pi$ be a Whittaker-type representation of $\GL_2(E)$. Then the space $\Hom_H(\pi, \triv)$ has dimension 1 if $\vep(\As(\pi))\omega_{A}(-1) = 1$ and is zero otherwise.
 \end{theorem}

 \begin{proof}The case of irreducible generic $\pi$ is proved in \cite{prasad92} assuming $\pi$ non-supercuspidal, and the supercuspidal case is filled in by \cite{prasadSP08}. In this case, the only example of a reducible Whittaker-type representation of $G$ is $\Sigma_E \otimes \eta$, where $\eta$ is a character of $E^\times$; and the central-character condition implies that $\lambda = \eta|_{F^\times}$ must be trivial or quadratic.

 The $\vep$-factors $\vep(\As(\St_E) \times \lambda)$ are computed in \cite[\S 8]{prasad92}. We find that $\vep(\As(\Sigma_E) \times \lambda)\omega_{E/F}(-1)$ is always $+1$. On the other hand, $\vep(\As(\St_E) \times \lambda) \omega_{E/F}(-1)$ is $+1$ if $\lambda$ is nontrivial quadratic, and $-1$ if $\lambda = 1$. So it follows that exactly one of $\Hom_H(\triv, \lambda)$ and $\Hom_H(\St_E, \lambda)$ is non-zero, implying that $\dim \Hom_H(\Sigma_E \otimes \eta, \triv) \le 1$.

 To complete the proof, we must show that when $\lambda \ne 1$, the $H$-invariant homomorphism $\Hom_H(\St_E, \lambda)$ extends to $\Sigma_E$. However, this is clear since the obstruction lies in $\operatorname{Ext}^1_H(\triv, \lambda)$, which is zero.
 \end{proof}

 This completes the proof of the Main Theorem.

\section{An application to Euler systems}

 We now give a global application, a strengthening of some results from \cite{leiloefflerzerbes18} and \cite{grossi19} on Euler systems for quadratic Hilbert modular forms. Let $K / \QQ$ be a real quadratic field and write $G = \operatorname{Res}_{K/\QQ}(\GL_2)$ and $H = \GL_{2/\QQ} \subset G$; set $G_f = G(\AA_{f}) = \GL_2(\AA_{K, f})$ and $H_f$ similarly.

 \subsection{Adelic representations}

 Let $\chi$ be a finite-order character of $\AA_f^\times$ and define a representation of $H_f$ by $\cI(\chi) = \sideset{}{'}\bigotimes_\ell \cI_\ell(\chi_\ell)$, where $\cI_{\ell}(\chi_{\ell})$ denotes the representation of $H_\ell$ given by normalised induction of the character $\chi_{\ell} |\cdot|^{1/2} \boxtimes |\cdot|^{-1/2}$ of the Borel subgroup. For $\chi = 1$, we let $\cI^0(1)$ denote the codimension 1 subrepresentation of $\cI(1)$. Exactly as in \cite[\S 2]{harrisscholl01}, the local results above imply the following branching law for $G_f$-representations:

 \begin{proposition}
 \label{prop:semilocal}
  Let $\pi$ be an irreducible admissible representation of $G_f$, all of whose local factors are generic, with $\omega_{\pi} |_{\AA_f^\times} = \chi^{-1}$.
  \begin{itemize}
   \item We have $\dim \Hom_{H_f}(\pi \otimes \cI(\chi), \triv) = 1$.
   \item If $\chi = 1$ and there exists some $\ell$ such that $\Hom_{H_\ell}(\pi_\ell, \triv) = 0$, then $\dim \Hom_{H_f}(\pi \otimes \cI^0(1), \triv) = 1$ and the natural restriction map $\Hom_{H_f}(\pi \otimes \cI(1), \triv) \to \Hom_{H_f}(\pi \otimes \cI^0(1), \triv)$ is a bijection.
   \item If $\chi = 1$ and $\Hom_{H_\ell}(\pi_\ell, \triv) \ne 0$ for all $\ell$, then $\dim \Hom_{H_f}(\pi \otimes \cI^0(1), \triv) = \infty$.
  \end{itemize}
 \end{proposition}

 \subsection{Hilbert modular forms}

  Suppose now that $\pi$ is (the finite part of) a cuspidal automorphic representation, arising from a Hilbert modular cusp form of parallel weight $k+2 \ge 2$, normalised so that $\omega_{\pi}$ has finite order.

  \begin{proposition}\label{prop:giada}
   Suppose $\pi$ is not a twist of a base-change from $\GL_{2/\QQ}$. Then, for any Dirichlet character $\tau$, there exist infinitely many primes $\ell$ such that $\Hom_{H_\ell}\left(\pi_\ell \otimes \tau_\ell, \triv\right) = 0$.
  \end{proposition}

  \begin{proof}
   See \cite[Proposition 7.2.5]{grossi19}.
  \end{proof}

  There is a natural $H_f$-representation $\cO^\times(Y)_{\CC}$ of \emph{modular units}, where $Y$ is the infinite-level modular curve (the Shimura variety for $\GL_2$). Note that this representation is smooth, but not admissible. It fits into a long exact sequence
  \[ 0 \to (\QQ^{\mathrm{ab}})^\times \otimes \CC \to \cO^\times(Y)_{\CC} \to \cI^0(1) \oplus \bigoplus_{\eta \ne 1}\cI(\eta) \to 0,\]
  with $H_f$ acting on $(\QQ^{\mathrm{ab}})^\times$ via the Artin reciprocity map of class field theory, and the sum is over all even Dirichlet characters $\eta$.

  There is a canonical homorphism, the \emph{Asai--Flach map}, constructed in \cite{leiloefflerzerbes18} (building on several earlier works such as \cite{Kings-Higher-Regulators}):
  \[ \mathcal{AF}^{[\pi, k]}: \left( \pi \otimes \cO^\times(Y)_{\CC} \right)_{H_f} \to H^1(\QQ, V^{\mathrm{As}}(\pi)^*(-k)),  \]
  where $V^{\mathrm{As}}(\pi)$ is the Asai Galois representation attached to $\pi$, and we have fixed an isomorphism $\overline{\QQ}_p \cong \CC$. The subscript $H_f$ indicates $H_f$-coinvariants.

  \begin{theorem}
   Suppose $\pi$ is not a twist of a base-change from $\QQ$. Then the Asai--Flach map factors through $\pi \otimes \cI(\chi)$, and its image is contained in a 1-dimensional subspace of $H^1(\QQ, V^{\mathrm{As}}(\pi)^*(-k))$.
  \end{theorem}

 \begin{proof}
  Using \cref{prop:giada}, we see that $\mathcal{AF}^{[\pi, k]}$ must vanish on $(\QQ^{\mathrm{ab}})^\times \otimes \CC$, so it factors through $\pi \otimes \cI(\chi)$ if $\chi \ne 1$, or $\pi \otimes \cI^0(\chi)$ if $\chi = 1$, where $\chi = (\omega_{\pi} |_{\AA_f^\times})^{-1}$ as above. Using \cref{prop:semilocal}, combined with a second application of \cref{prop:giada} if $\omega_{\pi}$ is trivial on $\QQ$, the result follows.
 \end{proof}

 As in the $\GSp_4$ case described in \cite[\S 6.6]{LZ20b-regulator}, one can remove the dependency on the test data entirely: using zeta-integrals, we can construct a canonical basis vector $Z_{\mathrm{can}} \in \Hom(\pi_f \otimes \cI(\chi), \triv)$, and define $\mathcal{AF}_{\mathrm{can}}^{[\pi, k]} \in H^1\left(\QQ, V^{\mathrm{As}}(\pi)^*(-k)\right)$ as the unique class such that $\mathcal{AF}^{[\pi, k]} = Z_{\mathrm{can}} \cdot \mathcal{AF}_{\mathrm{can}}^{[\pi, k]}$. We hope that this perspective may be useful in formulating and proving explicit reciprocity laws in the Asai setting.

 \begin{remark}
  The constructions of \cite{leiloefflerzerbes18} also apply to other twists of $V^{\mathrm{As}}(\pi)$, and to Hilbert modular forms of non-parallel weight; but in these other cases the input data for the Asai--Flach map lies in an irreducible principal series representation of $H_f$, so the necessary multiplicity-one results are standard. (The delicate cases are those which correspond to near-central values of $L$-series.)
 \end{remark}
 \providecommand{\bysame}{\leavevmode\hbox to3em{\hrulefill}\thinspace}
 \providecommand{\MR}[1]{}
 \renewcommand{\MR}[1]{%
  MR \href{http://www.ams.org/mathscinet-getitem?mr=#1}{#1}.
 }
 \providecommand{\href}[2]{#2}
 \newcommand{\articlehref}[2]{\href{#1}{#2}}

\end{document}